\documentclass[a4paper]{amsart}
\usepackage{amsmath,amssymb,amsxtra,amsfonts,amsthm,comment,url}

\usepackage{tikz}
\usetikzlibrary{arrows,backgrounds,decorations,automata}

\theoremstyle{plain}
\newtheorem{theorem}{Theorem}
\newtheorem{lemma}[theorem]{Lemma}
\newtheorem{corollary}[theorem]{Corollary}
\newtheorem{proposition}[theorem]{Proposition}

\theoremstyle{definition}
\newtheorem{remark}[theorem]{Remark}

\newcommand{\B}{\mathbb}
\newcommand{\C}{\mathcal}

\newcommand{\eps}{\varepsilon}

\begin{document}

\title{Regular sequences and the joint spectral radius}

\author{Michael Coons}
\address{School of Math.~and Phys.~Sciences\\
University of Newcastle\\
Callaghan\\
Australia}
\email{Michael.Coons@newcastle.edu.au}

\thanks{The research of M.~Coons was supported by ARC grant DE140100223.}

\date{\today}

%%%%%%%%%%%%%%%%%%%%%%%%%%%%%%%%%%%%%%%%%%%%%%%%
\begin{abstract} We classify the growth of a $k$-regular sequence based on information from its $k$-kernel. In order to provide such a classification, we introduce the notion of a growth exponent for $k$-regular sequences and show that this exponent is equal to the joint spectral radius of any set of a special class of matrices determined by the $k$-kernel.
\end{abstract}
%%%%%%%%%%%%%%%%%%%%%%%%%%%%%%%%%%%%%%%%%%%%%%%%

\maketitle

\vspace{-.9cm}
%%%%%%%%%%%%%%%%%%%%%%%%%%%%%%%%%%%%%%%%%%%%%%%%
\section{Introduction}
%%%%%%%%%%%%%%%%%%%%%%%%%%%%%%%%%%%%%%%%%%%%%%%%

Let $\B{K}$ be a field of characteristic zero. The {\em $k$-kernel} of $f:\B{Z}_{\geqslant 0}\to\B{K}$ is the set $${\rm Ker}_k(f):=\left\{\{f(k^\ell n+r)\}_{n\geqslant 0}:\ell\geqslant 0, 0\leqslant r< k^\ell\right\}.$$ The sequence $f$ is called {\em $k$-automatic} provided the set ${\rm Ker}_k(f)$ is finite \cite{C1972}. In 1992, as a generalisation of automatic sequences, Allouche and Shallit \cite{AS1992} introduced the notion of regular sequences. By their definition, a sequence $f$ taking values in $\B{K}$ is called {\em $k$-regular}, for an integer $k\geqslant 1$, provided the $\B{K}$-vector space $\langle {\rm Ker}_k(f)\rangle_{\B{K}}$ spanned by ${\rm Ker}_k(f)$ is finite dimensional. Connecting regular sequences to finite sets of matrices, Allouche and Shallit \cite[Lemma~4.1]{AS1992} showed that {\em a $\B{K}$-valued sequence $f$ is $k$-regular if and only if there exist a positive integer $d$, a finite set of matrices $\C{A}_f=\{{\bf A}_0,\ldots,{\bf A}_{k-1}\}\subseteq\B{K}^{d\times d}$, and vectors ${\bf v},{\bf w}\in \B{K}^d$ such that $f(n)={\bf w}^T {\bf A}_{i_0}\cdots{\bf A}_{i_s} {\bf v},$ where $(n)_k={i_s}\cdots {i_0}$ is the base-$k$ expansion of $n$.} Moreover, their proof showed that all such collections of matrices can be described (or constructed) by considering spanning sets of $\langle {\rm Ker}_k(f)\rangle_{\B{K}}$.

In their seminal paper, Allouche and Shallit \cite[Theorem~2.10]{AS1992} proved that given a $k$-regular sequence $f$, there is a positive constant $c_f$ such that $f(n)=O(n^{c_f})$.

In this paper, we determine the optimal value of the constant $c_f$. To state our result, we require a few definitions. Let $k\geqslant 1$ be an integer and $f:\B{Z}_{\geqslant 0}\to \B{K}$ be a (not eventually zero) $k$-regular sequence. We define the {\em growth exponent of $f$}, denoted ${\rm GrExp}(f)$, by $${\rm GrExp}(f):=\limsup_{\substack{n\to\infty\\ f(n)\neq 0}} \frac{\log|f(n)|}{\log n}.$$ The {\em joint spectral radius} of a finite set of matrices $\C{A}=\{{\bf A}_0,{\bf A}_1,\ldots, {\bf A}_{k-1}\}$, denoted $\rho(\C{A})$, is defined as the real number $$\rho(\C{A})=\limsup_{n\to\infty}\max_{0\leqslant i_0,i_1,\ldots,i_{n-1}\leqslant k-1}\left\| {\bf A}_{i_0}{\bf A}_{i_1}\cdots{\bf A}_{i_{n-1}}\right\|^{1/n},$$ where $\|\cdot\|$ is any (submultiplicative) matrix norm. This quantity was introduced by Rota and Strang \cite{RS1960} and has a wide range of applications. For an extensive treatment, see Jungers's monograph \cite{J2009}.

\begin{theorem}\label{main} Let $k\geqslant 1$ and $d\geqslant 1$ be integers and $f:\B{Z}_{\geqslant 0}\to \B{K}$ be a (not eventually zero) $k$-regular sequence. If $\C{A}_f$ is any collection of $k$ integer matrices associated to a basis of the $\B{K}$-vector space $\langle{\rm Ker}_k(f)\rangle_{\B{K}}$, then $$\log_k\rho(\C{A}_f)={\rm GrExp}(f),$$ where $\log_k$ denotes the base-$k$ logarithm.
\end{theorem}

We note that Theorem \ref{main} holds for $\B{K}$ replaced by any N\"otherian ring $R$, where $\C{A}_f$ is any collection of $k$ matrices associated to an $R$-module basis of the $R$-module spanned by ${\rm Ker}_k(f)$, where this $R$-module is viewed as an $R$-submodule of the set of a sequences with entries in $R$. In particular, the result holds for the ring $\B{Z}$.

\begin{remark} In engineering circles, for certain choices of $\C{A}$ related to a set $D$ of forbidden sign patterns, the quantity $\log_2 \rho(\C{A})$ is sometimes referred to as the {\em capacity} of the set $D$, denoted ${\rm cap}(D).$ See Jungers, Blondel, and Protasov \cite[Section II]{BJP2006} for details.
\end{remark}

%%%%%%%%%%%%%%%%%%%%%%%%%%%%%%%%%%%%%%%%%%%%%%%%
\section{The growth exponent of a regular sequence}
%%%%%%%%%%%%%%%%%%%%%%%%%%%%%%%%%%%%%%%%%%%%%%%%

In this section, all matrices are assumed to have entries in $\B{K}$ and all regular sequences are supposed to not eventually be zero.

\begin{lemma}\label{specgen} Let $k\geqslant 1$ be an integer and $\C{A}=\{{\bf A}_0,{\bf A}_1,\ldots,{\bf A}_{k-1}\}$ be a finite set of matrices. Given $\eps>0$ then there is a submultiplicative matrix norm $\|\cdot\|$ such that $\|{\bf A}_i\|<\rho(\C{A})+\eps$ for each $i\in\{0,1,\ldots,k-1\}$.
\end{lemma}

Lemma \ref{specgen} can be found in Blondel et al.~\cite[Proposition 4]{BNT2005}, though it was first given in the original paper of Rota and Strang \cite{RS1960}. 

\begin{proposition}\label{upper} Let $k\geqslant 2$ be an integer and $f:\B{Z}_{\geqslant 0}\to\B{K}$ be a $k$-regular function. For any $\eps>0$, there is a constant $c=c(\eps)>0$ such that for all $n\geqslant 1$, $$\frac{|f(n)|}{n^{\log_k(\rho(\C{A}_f)+\eps)}}\leqslant c,$$ where $\C{A}_f$ is the set any set of matrices associated to a spanning set of $\langle {\rm Ker}_k(f)\rangle_\B{K}$.\end{proposition}

\begin{proof} Let $\eps>0$ be given and let $\|\cdot\|$ be a matrix norm such that the conclusion of Lemma \ref{specgen} holds. Then $$|f(n)|\leqslant \|{\bf v}\|\cdot\|{\bf w}\|\cdot\prod_{j=0}^s\|{\bf A}_{i_j}\|\leqslant \|{\bf v}\|\cdot\|{\bf w}\|\cdot(\rho(\C{A})+\eps)^s,$$ where the base-$k$ expansion of $n$ is $i_s\cdots i_0.$ Using the bound $s\leqslant \log_k n$ with some rearrangement gives the result.
\end{proof}

\begin{lemma}\label{mAm} Let $k\geqslant 1$ be an integer and $\C{A}=\{{\bf A}_0,{\bf A}_1,\ldots,{\bf A}_{k-1}\}$ be a finite set of matrices. If $\eps>0$ is a real number, then there is a positive integer $m$ and a matrix ${\bf A}_{i_0}\cdots{\bf A}_{i_{m-1}}$, such that $$(\rho(\C{A})-\eps)^m<\rho({\bf A}_{i_0}\cdots{\bf A}_{i_{m-1}})<(\rho(\C{A})+\eps)^m.$$
\end{lemma}

\begin{proof} This is a direct consequence of the definition of the joint spectral radius.
\end{proof}

Now let $k\geqslant 2$ be an integer, and suppose that $f:\B{Z}_{\geqslant 0}\to \mathbb{K}$ is an unbounded $k$-regular sequence. Given a word $w=i_s\cdots i_0\in \{0,\ldots ,k-1\}^*$, we let $[w]_k$ denote the natural number such that $(n)_k=w$.    Let $\{\{f(n)\}_{n\geqslant 0}=\{g_1(n)\}_{n\geqslant 0},\ldots ,\{g_d(n)\}_{n\geqslant 0}\}$ be a basis for the $\B{K}$-vector space $\langle{\rm Ker}_k(f)\rangle_{\B{K}}$.  Then for each $i\in \{0,1,\ldots ,k-1\}$, the sequences $\{g_1(kn+i)\}_{n\geqslant 0},\ldots ,\{g_d(kn+i)\}_{n\geqslant 0}$ can be expressed as $\mathbb{K}$-linear combinations of $\{g_1(n)\}_{n\geqslant 0},\ldots ,\{g_d(n)\}_{n\geqslant 0}$ and hence there is a set of $d\times d$ matrices $\C{A}_f=\{{\bf A}_0,\ldots ,{\bf A}_{k-1}\}$ with entries in $\B{K}$ such that $$ {\bf A}_i[g_1(n),\ldots ,g_d(n)]^T =[g_1(kn+i),\ldots ,g_d(kn+i)]^T$$ for $i=0,\ldots ,k-1$ and all $n\geqslant 0$.  In particular, if $i_s\cdots i_0$ is the base-$k$ expansion of $n$, then $ {\bf A}_{i_0}\cdots {\bf A}_{i_s}[g_1(0),\ldots ,g_d(0)]^T = [g_1(n),\ldots ,g_d(n)]^T$.   (We note that this holds even if we pad the base-$k$ expansion of $n$ with zeros at the beginning.) We call such a set of matrices $\C{A}_f$, constructed in this way, {\em a set of matrices associated to a basis of $\langle {\rm Ker}_k(f)\rangle_\B{K}$}.

This construction allows us to provide a lower bound analogue of Proposition \ref{upper}.

\begin{proposition}\label{lower} Let $k\geqslant 2$ be an integer and $f:\B{Z}_{\geqslant 0}\to\B{K}$ be a $k$-regular function. For any $\eps>0$, there is a constant $c=c(\eps)>0$ such that for infinitely many $n\geqslant 1$, $$\frac{|f(n)|}{n^{\log_k(\rho(\C{A}_f)-\eps)}}\geqslant c,$$ where $\C{A}_f$ is any set of matrices associated to a basis of $\langle {\rm Ker}_k(f)\rangle_\B{K}$.
\end{proposition}

\begin{proof} 
Let $\eps>0$ be given. Then by Lemma \ref{mAm} there is a positive integer $m$ and a matrix ${\bf A}={\bf A}_{i_0}\cdots{\bf A}_{i_{m-1}}$ such that $\rho({\bf A})>(\rho(\C{A}_f)-\eps)^m.$ Let $\lambda$ be an eigenvalue of ${\bf A}$ with $|\lambda|=\rho({\bf A}).$ Then there is an eigenvector ${\bf y}$ such that ${\bf A}{\bf y}=\lambda{\bf y}.$ Pick a vector ${\bf x}$ such that ${\bf x}^T{\bf y}=c_1\neq 0.$ Then $$\left|{\bf x}^T{\bf A}^n{\bf y}\right|=|c_1|\cdot|\lambda|^n=|c_1|\cdot\rho({\bf A})^n>|c_1|\cdot\left(\rho(\C{A}_f)-\eps\right)^{nm}.$$

Using a method developed by Bell, Coons, and Hare \cite{BCH2014}, it follows (see Appendix \ref{Appendix} for details) that there are words $u_1,\ldots,u_d,v_1,\ldots,v_t$ from $\{0,1,\ldots,k-1\}^*$ and a positive constant $c_2$ such that for each $n\geqslant 0$ there is an element from $$\left\{|f([u_i(i_{m-1}\cdots i_0)^nv_j]_k)|:i=1,\ldots,d,\ j=1,\ldots,t\right\},$$ which is at least $c_2(\rho(\C{A}_f)-\eps)^{nm}.$ Here, as previously, we have used the notation $[w]_k$ to be the integer $n$ such that $(n)_k=w$.

If $M=\max\{|u_i|,|v_j|:i=1,\ldots,d,\ j=1,\ldots,t\}$, then $$N=[u_i(i_{m-1}\cdots i_0)^nv_j]_k<k^{2M+nm},$$ so that $\log_k(N)-2M<nm.$ Thus, by the finding of the previous paragraph, there are infinitely many $N$ such that $$\frac{|f(N)|}{N^{\log_k(\rho(\C{A}_f)-\eps)}}=\frac{|f(N)|}{(\rho(\C{A}_f)-\eps)^{\log_kN}}>\frac{c_2}{(\rho(\C{A}_f)-\eps)^{2M}},$$ which is the desired result.
\end{proof}

\begin{proof}[Proof of Theorem \ref{main}] For a given $\eps>0$, Proposition \ref{upper} implies that $$\lim_{n\to\infty}\frac{|f(n)|}{n^{\log_k(\rho(\C{A}_f)+2\eps)}}=0,$$ and Proposition \ref{lower} implies that $$\limsup_{n\to\infty} \frac{|f(n)|}{n^{\log_k(\rho(\C{A}_f)-2\eps)}}=\infty.$$ Taken together these give $$\log_k(\rho(\C{A}_f)-2\eps)\leqslant {\rm GrExp}(f)\leqslant \log_k(\rho(\C{A}_f)+2\eps).$$ Since $\eps$ can be taken arbitrarily small, this proves the theorem.
\end{proof}

We end this section by highlighting one major difference between Proposition~\ref{upper} and Proposition \ref{lower}. Proposition \ref{upper} is true for $\C{A}_f$ related to any spanning set of the $\B{K}$-vector space $\langle {\rm Ker}_k(f)\rangle_\B{K}$, while Proposition \ref{lower} requires $\C{A}_f$ to be associated to a basis of $\langle {\rm Ker}_k(f)\rangle_\B{K}$. In fact, these two propositions give the following corollary.

\begin{corollary} Let $k\geqslant 2$ be an integer and $f:\B{Z}_{\geqslant 0}\to\B{K}$ be a $k$-regular function. If $\C{B}_f$ is any set of matrices associated to $f$ and $\C{A}_f$ is any set of matrices associated to a basis of $\langle {\rm Ker}_k(f)\rangle_\B{K}$, then $\rho(\C{A}_f)\leqslant\rho(\C{B}_f).$
\end{corollary}

Equality in the conclusion of the above corollary would be desirable, but unfortunately, this is not (in general) the case. To see this, consider the $2$-regular function $f$, where, for $(n)_2=i_s\cdots i_0$, we have $f(n)={\bf w}^T {\bf A}_{i_0}\cdots{\bf A}_{i_s} {\bf v}$, with $$\C{A}_f=\{{\bf A}_0,{\bf A}_1\}=\left\{\left(\begin{matrix} 1 & 0\\ 0 & 1\end{matrix}\right)\right\},\quad {\bf w}^T=[1\ 0],\quad\mbox{and}\quad{\bf v}=[1\ 1]^T.$$ Then also for any number $x>1$, we have $f(n)={\bf x}^T {\bf B}_{i_0}\cdots{\bf B}_{i_s} {\bf y}$, with $$\C{B}_f=\{{\bf B}_0,{\bf B}_1\}=\left\{\left(\begin{matrix} 1&0&0\\ 0&1&0\\ 0&0&x\end{matrix}\right)\right\},\quad {\bf x}^T=[1\ 0\ 0],\quad\mbox{and}\quad{\bf y}=[1\ 1\ 0]^T,$$ and $$\rho(\C{A}_f)=1<x=\rho(\C{B}_f).$$

%%%%%%%%%%%%%%%%%%%%%%%%%%%%%%%%%%%%%%%%%%%%%%%%
\appendix\section{{}}\label{Appendix}
%%%%%%%%%%%%%%%%%%%%%%%%%%%%%%%%%%%%%%%%%%%%%%%%

For a given $\eps>0$, we had by Lemma \ref{mAm} that there is a positive integer $m$ and a matrix ${\bf A}={\bf A}_{i_0}\cdots{\bf A}_{i_{m-1}}$ such that $\rho({\bf A})>(\rho(\C{A}_f)-\eps)^m.$ Choosing an eigenvalue $\lambda$ of ${\bf A}$ with $|\lambda|=\rho({\bf A}),$ we found vectors ${\bf x}$ and ${\bf y}$ such that ${\bf x}^T{\bf y}=c_1\neq 0$ and \begin{equation}\label{xAy} \left|{\bf x}^T{\bf A}^n{\bf y}\right|=|c_1|\cdot|\lambda|^n=|c_1|\cdot\rho({\bf A})^n>|c_1|\cdot\left(\rho(\C{A}_f)-\eps\right)^{nm}.\end{equation}

In this appendix, we follow an argument of Bell, Coons, and Hare \cite[p.~198]{BCH2014} to provide the existence of words $u_1,\ldots,u_d,v_1,\ldots,v_t$ from $\{0,1,\ldots,k-1\}^*$ such that for each $n\geqslant 0$ there is an element from $$\left\{|f([u_i(i_{m-1}\cdots i_0)^nv_j]_k)|:i=1,\ldots,d,\ j=1,\ldots,t\right\},$$ which is at least $c_2(\rho(\C{A}_f)-\eps)^{nm}.$

To this end, let $k\geqslant 2$ be an integer, suppose that $f:\B{Z}_{\geqslant 0}\to \mathbb{K}$ is an unbounded $k$-regular sequence, and $\C{A}_f=\{{\bf A}_0,\ldots,{\bf A}_{k-1}\}$ be a set of matrices associated to a basis $\{\{f(n)\}_{n\geqslant 0}=\{g_1(n)\}_{n\geqslant 0},\ldots ,\{g_d(n)\}_{n\geqslant 0}\}$ of the $\B{K}$-vector space $\langle {\rm Ker}_k(f)\rangle_\B{K}$.

We claim that the $\mathbb{K}$-span of the vectors $[g_1(i),\ldots ,g_d(i)]$, as $i$ ranges over all natural numbers, must span all of $\mathbb{K}^d$. If this were not the case, then their span would be a proper subspace of $\mathbb{K}^d$ and hence the span would have a non-trivial orthogonal complement.  In particular, there would exist $c_1,\ldots ,c_d\in\B{K}$, not all zero, such that $$c_1g_1(n)+\cdots +c_dg_d(n)=0$$ for every $n$, contradicting the fact that $g_1(n),\ldots ,g_d(n)$ are $\B{K}$-linearly independent sequences. 

Let $\langle\mathcal{A}_f\rangle$ denote the semigroup generated by the elements of $\C{A}_f$.  We have just shown that there exist words ${\bf X}_1,\ldots ,{\bf X}_d$ in $\langle\mathcal{A}_f\rangle$ such that $$  [g_1(0),\ldots ,g_d(0)] {\bf X}_1,\ldots , [g_1(0),\ldots ,g_d(0)]{\bf X}_d$$ span $\mathbb{K}^d$.    

Now consider ${\bf x}^T{\bf A}^n{\bf y}$ as described in the first paragraph of this appendix. By construction, we may write ${\bf x}^T=\sum_j \alpha_j [g_1(0),\ldots ,g_d(0)] {\bf X}_j$ for some complex numbers $\alpha_j$.  
Then 
$${\bf x}^T {\bf A}^n = \sum_{j} \alpha_j [g_1(0),\ldots ,g_d(0)]{\bf X}_j {\bf A}^n .$$     Let $u_j$ be the word in $\{0,1,\dots ,k-1\}^*$ corresponding to ${\bf X}_j$ and let $y=i_s\cdots i_0$ be the word in $\{0,\ldots ,k-1\}^*$ corresponding to ${\bf A}$; that is $y=i_s\cdots i_0$ where ${\bf A} = {\bf A}_{i_s}\cdots {\bf A}_{i_0}$ and similarly for $u_j$. Then we have $$[g_1(0),\ldots ,g_d(0)]{\bf X}_j{\bf A}^n=[g_1([u_j y^n]_k),\ldots ,g_d([u_j y^n]_k)]^T.$$

Write ${\bf y}^T=[\beta_1,\ldots ,\beta_d]$.  Then $${\bf x}^T {\bf A}^n {\bf y} =\sum_{i,j} \alpha_i\beta_j g_j([u_i y^n]_k).$$
By assumption, each of $\{g_1(n)\}_{n\geqslant 0},\ldots ,\{g_d(n)\}_{n\geqslant 0}$ is in the $\mathbb{K}$-vector space generated by ${\rm Ker}_k(f)$, and hence there exist natural numbers $p_1,\ldots ,p_t$ and $q_1,\ldots ,q_t$ with $0\leqslant q_m <k^{p_m}$ for $m=1,\ldots ,t$ such that each of for $s=1,\ldots ,d$, we have
$g_s(n)= \sum_{i=1}^t  \gamma_{i,s} f(k^{p_i} n + q_i)$ for some constants $\gamma_{i,s}\in\B{K}$.  
Then 
$${\bf x}^T{\bf A}^n {\bf y} = \sum_{i,j,\ell} \alpha_i \beta_j \gamma_{\ell,j} f([u_i y^n v_{\ell}]_k),$$ where $v_\ell$ is the unique word in $\{0,1,\ldots ,k-1\}^*$ of length $p_\ell$ such that $[v_\ell]_k=q_\ell$.  Let $K=\sum_{i,j,\ell} |\alpha_i|\cdot |\beta_j|\cdot |\gamma_{\ell,j}|$.  Then since
$|{\bf x}^T {\bf A}^n {\bf y}|\geqslant |c_1|\cdot\left(\rho(\C{A}_f)-\eps\right)^{nm}$ for all $n$,  some element from
$$\big\{ |f([u_i y^n v_j]_k)| : i=1,\ldots, d, j=1,\ldots ,t\}\big\}$$ is at least $(|c_1|/K)\cdot\left(\rho(\C{A}_f)-\eps\right)^{nm}$ for each $n$. \\ 

\noindent{\bf Acknowledgements.} We thank Bj\"{o}rn R\"{u}ffer for several useful conversations.

%%%%%%%%%%%%%%%%%%%%%%%%%%%%%%%%%%%%%%%%%%%%%%%%
\bibliographystyle{amsplain}
\providecommand{\bysame}{\leavevmode\hbox to3em{\hrulefill}\thinspace}
\providecommand{\MR}{\relax\ifhmode\unskip\space\fi MR }
% \MRhref is called by the amsart/book/proc definition of \MR.
\providecommand{\MRhref}[2]{%
  \href{http://www.ams.org/mathscinet-getitem?mr=#1}{#2}
}
\providecommand{\href}[2]{#2}

%%%%%%%%%%%%%%%%%%%%%%%%%%%%%%%%%%%%%%%%%%%%%%%%

\end{document}